\documentclass[10pt,twoside]{amsart}

\usepackage{graphicx}
\usepackage{amssymb}
\usepackage{verbatim}
\usepackage[T1]{fontenc}
 \usepackage{pifont}
\usepackage[usenames, dvipsnames]{color}

\numberwithin{equation}{section}
\newtheorem{theorem}{Theorem}[section] 
\newtheorem{lemma}[theorem]{Lemma}
\newtheorem{proposition}[theorem]{Proposition}

\theoremstyle{definition}
\newtheorem{definition}[theorem]{Definition}
\newtheorem{remark}[theorem]{Remark}
\newtheorem{example}[theorem]{Example}

\newcommand{\popo}{\mathbb{P}^1 \times \mathbb{P}^1}

\newcommand{\N}{\mathbb{N}}

\newcommand{\Si}{\mathcal{S}}

\newcommand{\M}{\mathcal{M}}
\newcommand{\Li}{\mathcal{L}}

\renewcommand{\thetheorem}{\thesection.\arabic{theorem}}
\begin{document}
\title[The Hilbert function of bigraded Algebras ]{The Hilbert function of bigraded Algebras in $k[\popo]$ }

\author{Giuseppe Favacchio}
\address{Dipartimento di Matematica e Informatica\\
Viale A. Doria, 6 - 95100 - Catania, Italy}
\email{favacchio@dmi.unict.it} \urladdr{www.dmi.unict.it/~gfavacchio}




\begin{abstract} We classify the Hilbert functions of bigraded algebras in $k[x_1,x_2,y_1,y_2]$ by introducing a numerical function called a $Ferrers$ function.  
\end{abstract}

\maketitle


\renewcommand{\thetheorem}{\thesection.\arabic{theorem}}

\setcounter{section}{0}
\setcounter{theorem}{0}
\section{Introduction}
	
	Let $S:=k[x_1,\ldots, x_n]$ be the standard graded polynomial ring and let $I\subseteq S$ be a homogeneous ideal. The quotient ring  $S/I$ is called a \textit{standard graded k-algebra}. The Hilbert function of  $S/I$ is defined as $H_{S/I}: \mathbb N\to \mathbb N $ such that
	\[H_{S/I}(t):=\dim_k\left(S/I\right)_t= \dim_k S_t - \dim_k I_t. \]
	
	A famous theorem, due to Macaulay (cf. \cite{M}) and pointed out by Stanley (cf. \cite{St}), characterizes the numerical functions that are Hilbert functions of a standard graded $k$-algebra, i.e. the functions $H$ such that $H=H_{S/I}$ for some homogeneous ideal $I\subseteq S$. Macaulay's theorem is expressed in the language of $O$-sequence; for a modern treatment of this result see \cite{BH}.

	It is of interest to find an extension of the above theorem to the multi-graded case.
	Multi-graded Hilbert functions arise in many contexts. Properties related to the Hilbert function of multi-graded algebras are studied for instance in \cite{ACD, BG, PS, SVT,TV}. 
	A generalization of Macaulay's theorem to multi-graded rings is an open problem. 
	A partial result is Theorem 4.14 in \cite{ACD}. It gives non-sharp bounds on the growth of the Hilbert function of a bigraded algebra. 
	
	The goal of this work is to generalize the Macaulay's Theorem in the first significant case of bigraded algebras. The main result of this paper is Theorem \ref{main}, where we classify the numerical functions $H:\N^2\to \N$ which are Hilbert functions of a bigraded algebra in $k[x_1,x_2,y_1,y_2]$ where $\deg(x_i)=(1,0)$ and $\deg(y_j)=(0,1).$ 
	
	The paper is structured as follows. 
	In Section \ref{sec:ACD} we give the necessary background and notation. In Section \ref{sec:Partitions} we introduce a set of partitions of a number, and we define a numerical function called a $Ferrers$ function. Finally, in Section \ref{sec:main} we investigate a connection between partitions and set of monomials, and we prove that Ferrers functions characterize the Hilbert functions of bigraded algebras.   

	\textbf{Acknowledgment.} I would like to thank Aldo Conca for introducing me to this topic. I also would like to express my gratitude to Elena Guardo and Alfio Ragusa for the useful suggestions. The computer program CoCoA \cite{C} was indispensable for all the computations. I also thank the referee for his/her useful comments.

\section{Maximal growths for the Hilbert function of a bigraded algebra}\label{sec:ACD}
	
	Let $k$ be an infinite field, and let $R :=k[x_1, x_2 , y_1, y_2 ] =k[\popo]$ be the polynomial ring in 4 indeterminates with the grading defined by $\deg x_i = (1, 0)$ and $\deg y_j = (0, 1).$ Then $R=\oplus_{(i,j)\in \N^2}R_{(i,j)}$ where $R_{(i,j)}$ denotes the set of all homogeneous elements in $R$ of degree $(i,j).$  $R_{(i,j)}$ is generated, as a $k$-vector space, by the monomials $x_1^{i_1}x_2^{i_2}y_1^{j_1}y_2^{j_2}$ such that $i_1+i_2=i$ and $j_1+j_2=j.$
	An ideal $I\subseteq R$ is called a $bigraded$ ideal if it is generated by homogeneous elements with respect to this grading. A bigraded algebra $R/I$ is the quotient of $R$ with a bigraded ideal $I.$  The Hilbert function of a bigraded algebra $R/I$ is defined such that $H_{R/I}:\N^2\to \N$ and 
	$H_{R/I}(i,j):=\dim_k (R/I)_{(i,j)}=\dim_k R_{(i,j)} -\dim_k I_{(i,j)}$ where $I_{(i,j)} =I\cap R_{(i,j)}$ is the set of the bihomogeneous elements of degree $(i,j)$ in $I.$

	Throughout this notes we will work with the degree lexicographical order on $R$ induced by $x_1 > x_2 > y_1 > y_2.$ With this ordering we recall the definition of bilex ideal, introduced and studied in \cite{ACD}. We refer to \cite{ACD} for all preliminaries and for further results on bilex ideals.

\begin{definition}[\cite{ACD}, Definition 4.4]\label{Def4.4} 
	A set of monomials $L\subseteq R_{(i,j)}$ is called $bilex$ if for every monomial $uv\in L,$ where $u\in R_{(i,0)}$ and $v\in R_{(0,j)},$  the following conditions are satisfied:
	\begin{itemize}
		\item if $u'\in R_{(i,0)}$ and $u' > u$, then $u'v \in L;$
		\item if $v'\in R_{(0,j)}$ and $v' > v$, then $uv' \in L.$
	\end{itemize}
	A monomial ideal $I \subseteq R$ is called a $bilex\ ideal$ if $I_{(i,j)}$ is generated as $k$-vector space by a bilex set of monomials, for every $i,j\ge 0$.
\end{definition}

Bilex ideals play a crucial role in the study of the Hilbert function of bigraded algebras.
\begin{theorem}[\cite{ACD},Theorem 4.14] Let $J\subseteq R$ be a bigraded ideal. Then there exists a bilex ideal $I$
	such that $H_{R/I}=H_{R/J}.$
\end{theorem}

The next theorem gives an upper bound for the growth of the Hilbert function of bigraded algebras. It is an  reformulation of \cite{ACD}, Theorem 4.18.
\begin{theorem}\label{Thm4.18}Let $I\subseteq R$ be a bigraded ideal. For all $(i,j)\in \mathbb N^2$, if $p:=\left\lfloor \frac{H_{R/I}(i,j)}{(i+1)}\right\rfloor$ and $q:=\left\lfloor \frac{H_{R/I}(i,j)}{(j+1)}\right\rfloor$, 
	then
	$$\left\{\begin{array}{l}
	H_{R/I}(i+1,j)\le H_{R/I}(i,j)+p \\
	H_{R/I}(i,j+1)\le H_{R/I}(i,j)+q
	\end{array}\right. $$
\end{theorem}

\begin{remark}\label{Es1}
	The bound in Theorem \ref{Thm4.18} is not sharp even if $\dim_k(I_{(i,j)})=  2.$ Take, for instance, $I':=(x_1y_1,x_1y_2)+ (x_1,x_2,y_1,y_2)^4$ and $I'':=(x_1y_1,x_2y_1)+ (x_1,x_2,y_1,y_2)^4.$  Then the Hilbert functions of the associated bigraded algebras are 
	$$ H_{R/I'}:= \begin{array}{l|llllll}
	  & 0 & 1 & 2 & 3 & 4 &\ldots\\
	  \hline
	 0  & 1 & 2 & 3 & 4 & 0 &\ldots\\
	 1  & 2 & 2 & 3 & 0 & 0 &\ldots\\
	 2  & 3 & 2 & 0 & 0 & 0 &\ldots\\
	 3  & 4 & 0 & 0 & 0 & 0 &\ldots\\
	 4  & 0 & 0 & 0 & 0 & 0 &\ldots\\       
	 \end{array}
	 \ \ \ \
	  H_{R/I''}:= \begin{array}{l|llllll}
	  & 0 & 1 & 2 & 3 & 4 &\ldots\\
	  \hline
	  0  & 1 & 2 & 3 & 4 & 0 &\ldots\\
	  1  & 2 & 2 & 2 & 0 & 0 &\ldots\\
	  2  & 3 & 3 & 0 & 0 & 0 &\ldots\\
	  3  & 4 & 0 & 0 & 0 & 0 &\ldots\\
	  4  & 0 & 0 & 0 & 0 & 0 &\ldots\\       
	  \end{array}.
	  $$
	  
	By Theorem \ref{Thm4.18} we have $\left\{\begin{array}{l}
	H_{R/I}(2,1)\le 3 \\
	H_{R/I}(1,2)\le 3
	\end{array}\right. $ but the numeric function
		$$ H:= \begin{array}{l|llllll}
		& 0 & 1 & 2 & 3 & 4 &\ldots\\
		\hline
		0  & 1 & 2 & 3 & 4 & 0 &\ldots\\
		1  & 2 & 2 & 3 & 0 & 0 &\ldots\\
		2  & 3 & 3 & 0 & 0 & 0 &\ldots\\
		3  & 4 & 0 & 0 & 0 & 0 &\ldots\\
		4  & 0 & 0 & 0 & 0 & 0 &\ldots\\       
		\end{array}
		$$
	is not the Hilbert function of any bigraded algebra. Indeed, for $t\ge 0,$   $\sum_{i+j=t}H(i,j)$ gives rise to the sequence $(1, 4, 8, 14, 0, \ldots )$ which fails to be an $O$-sequence. Thus, this example also shows two different maximal growths of the Hilbert function from the degree $(1,1)$ to $(1,2)$ and $(2,1).$ 
\end{remark}

  To formalize the idea of maximal growths we need to introduce a partial order in $\N^2.$ For $(a,b), (c,d)\in \N^2$ we say that $(a,b)\leq (c,d)$ iff $a\le c$ and $b\le d.$ Moreover we say that $(a,b)< (c,d)$ iff $(a,b)\leq (c,d)$ and  $a< c$ or $b< d.$

\begin{definition}Let $I\subseteq R$ be a bigraded ideal. We say that $H_{R/I},$ the Hilbert function of $R/I,$ has a maximal growth in degree $(i,j)$ if $(H_{R/I}(i+1,j), H_{R/I}(i,j+1))$ is a maximal element in the set $\{ (H_{R/J}(i+1,j), H_{R/J}(i,j+1)) \ |\ R/J\  \text{a bigraded algebra with}\  H_{R/J}(i,j)=H_{R/I}(i,j)  \}.$ 
\end{definition}

Note that the above definition does not require that $R/I$ and $R/J$ have the same Hilbert function in the degrees less than $(i,j)$.

\section{Partitions of a number and Ferrers functions}\label{sec:Partitions}
In this section we introduce a partition of a number which slightly generalizes Definition 3.12 in \cite{GVT} that allows the authors to characterize arithmetically Cohen-Macaulay sets of points in $\popo$. 
\begin{definition} 
	Given $h,\ell_1,\ell_2\in \N$ we say that $\alpha:=(p_1,\ldots, p_{t})\in \N^t$ is a partition of $h$  of $sides$ $(\ell_1,\ell_2)$ if $h=p_1+\cdots+p_{t},$  $\ell_1\ge p_1\ge \ldots\ge p_{t}\ge 0$ and $t=\ell_2.$
	We denote by $\lambda_1(\alpha)$ the number of entries in $\alpha$ equal to $\ell_1$, i.e. $\lambda_1(\alpha):=|\{j\ |\ p_j=\ell_1\}|$ and by $\lambda_2(\alpha):=p_{\ell_2}.$ Moreover we call $\lambda(\alpha):=(\lambda_1(\alpha),\lambda_2(\alpha))\in \N^2$ the $size$ of $\alpha.$
\end{definition}

For example, $\alpha=(5,5,5,4,1,1)$ is a partition of $21$ of sides $(5,6)$  with size $\lambda(\alpha)=(3,1)$.


We denote by $\Si(h)^{(\ell_1,\ell_2)}$ the set of all the partitions of $h$ of sides $(\ell_1,\ell_2),$ and by $\Li(h)^{(\ell_1,\ell_2)}$ the set of the sizes of the elements in  $\Si(h)^{(\ell_1,\ell_2)}.$ 

\begin{example}
	$\Si(4)^{(3,3)}:=\{(3,1,0),(2,2,0),(2,1,1)\}$ and $\Li(4)^{(3,3)}:=\{(1,0),(0,0),(0,1)\}.$
\end{example}

Furthermore, we set $\Si(\cdot)^{(\ell_1,\ell_2)}:=\cup_{h\in \N }\Si(h)^{(\ell_1,\ell_2)}.$ We introduce in $\Si(\cdot)^{(\ell_1,\ell_2)}$ an inner operation. Take $\alpha:=(p_1,\ldots, p_{t})$ and $\alpha':=(p'_1,\ldots, p'_{t})$ elements in $\Si(\cdot)^{(\ell_1,\ell_2)}$. Then we define 
$$\alpha \cap \alpha':= (\min\{p_1,p'_1\}, \ldots, \min\{p_t,p'_t\})\in \Si(\cdot)^{(\ell_1,\ell_2)}. $$

Let $\alpha, \alpha'\in \Si(\cdot)^{(\ell_1,\ell_2)}$. We say that $\alpha'\le \alpha$ iff $\alpha' \cap \alpha=\alpha',$ i.e.
the entries in $\alpha'$ are less than or equal to the entries in $\alpha$ componentwise. 

Let $\alpha:=(p_1,p_2,\ldots, p_{t})\in\Si(\cdot)^{(\ell_1,\ell_2)}$ be a partition of size $(\lambda_1,\lambda_2)$. Then we associate to $\alpha$ partitions of sides $(\ell_1+1,\ell_2)$ and $(\ell_1,\ell_2+1)$ as follows. 
We denote by $\alpha^{(1)}:=(p'_1,p'_2,\ldots, p'_{t})\in\Si(\cdot)^{(\ell_1+1,\ell_2)}$ where  $$p'_j:=\begin{cases}p_j+1 &\ \text{if}\ j\le\lambda_1\\
 p_j &\ \text{if}\ j>\lambda_1 \end{cases},$$
 and we denote by 
$\alpha^{(2)}:=(p_1,p_2,\ldots, p_{t-1}, p_{t}, p_{t})\in\Si(\cdot)^{(\ell_1,\ell_2+1)}.$

We are ready to introduce the Ferrers functions. 

\begin{definition}\label{DefFerrers} Let $H:\N\times\N\to \N$ be a numerical function. We say that $H$ is a Ferrers function if $H(0,0)=1$ and, for any $(i,j)\in \N^2,$ there exists  a partition of $H(i,j)$ of sides $(i+1, j+1),$ namely $\alpha_{ij}\in \Si{(H(i,j))}^{(i+1,j+1)},$ such that all the partitions satisfy the condition  
	$$\left\{\begin{array}{ll}
	\alpha_{ij}\le\alpha_{i-1j}^{(1)}& \text{if}\ i>0 \\
	\alpha_{ij}\le\alpha_{ij-1}^{(2)}& \text{if}\ j>0 \\
	\end{array}\right.$$
\end{definition}

\begin{example}
	Let $H:\mathbb N^2 \to \mathbb N$ be the numerical function
	$H(i,j)= (i+1)(j+1).$ For any $i,j\in \mathbb N,$ the only partition of sides $(i+1,j+1)$ of the integer $(i+1)(j+1)$ is  $\alpha_{ij}:=(\underbrace{i+1, i+1, \cdots, i+1}_{j+1})$. 
	For any $i,j\in \mathbb N$, we have $\alpha_{i-1j}^{(1)}=\alpha_{ij}$ and $\alpha_{ij-1}^{(2)}=\alpha_{ij}$. Then 
	the conditions in  Definition \ref{DefFerrers} are satisfied and therefore $H$ is a Ferrers function. 	
\end{example}

\begin{remark}\label{remG}
	Note that if $H$ is a Ferrers function then we have a bound on the growth of $H$ since 
	$$H(i,j)\le \min \{H(i-1,j)+\lambda_1(\alpha_{i-1j}), H(i,j-1)+\lambda_2(\alpha_{ij-1})\}.$$
\end{remark}

\begin{proposition}Let $H$ be a Ferrers function, set $g_1:=H(i+1,j)-H(i,j)$ and  $g_2:=H(i,j+1)-H(i,j).$ Then $(g_1,g_2)\le(\lambda_1,\lambda_2),$ for some $(\lambda_1,\lambda_2)\in \Li{(H(i,j))}^{(i+1,j+1)}.$
\end{proposition}
\begin{proof} Since $H$ is a Ferrers function then, from Remark \ref{remG}, there exists $\alpha_{ij}\in \Si{(H(i,j))}^{(i+1,j+1)},$ such that $(H(i+1,j),H(i,j+1))\le (H(i,j)+\lambda_1(\alpha_{ij}),H(i,j)+\lambda_2(\alpha_{ij}) ).$
Therefore $(g_1,g_2)\le (\lambda_1(\alpha_{ij}), \lambda_2(\alpha_{ij}) ).$	
\end{proof}

\begin{example}The numeric function
	$$ H:= \begin{array}{l|lllll}
	& 0 & 1 & 2 & 3 & \ldots\\
	\hline
	0  & 1 & 2 & 3 & 0 &\ldots\\
	1  & 2 & 2 & 3 & 0 & \ldots\\
	2  & 3 & 3 & 0 & 0 &\ldots\\
	3  & 0 & 0 & 0 & 0 &\ldots\\
       
	\end{array}
	$$
fails to be a Ferrers function since $\Si(2)^{(2,2)}=\{(2,0),(1,1)\}$ and $(1,1)\notin \Li(2)^{(2,2)}=\{(1,0),(0,1)\}.$	
\end{example}

\section{The Hilbert function of a bigraded algebra}\label{sec:main}
In order to relate Ferrers functions to bigraded algebras we introduce a correspondence between partitions and sets of monomials.

\begin{definition}\label{T(p,q)} 
	We denote by $T(p,q)_{(a,b)}\in R_{(a,b)},$ where $(0,0)\le (p,q)\le (a+1,b+1)\in \N^2,$ the monomial
	$$T(p,q)_{(a,b)}:=\begin{cases}
	x_1^{p-1}x_2^{a-p+1}y_1^{q-1}y_2^{b-q+1}& \text{if}\ (a,b),(p,q)\ge(1,1) \\
	y_1^{q-1}y_2^{b-q+1}& \text{if}\ a=0, b>0, (p,q)\ge(1,1) \\
	x_1^{p-1}x_2^{a-p+1}& \text{if}\ b=0, a>0, (p,q)\ge(1,1) \\
	1 & \text{if}\ a=b=0, (p,q)=(1,1) \\
	0 & \text{if}\ p=0\ \text{or}\ q=0
	\end{cases} $$ 

	Given $\alpha:=(p_1,\ldots, p_{t})\in \Si(h)^{(a+1,b+1)},$ a partition of an integer $h$ of sides $(a+1, b+1),$ we denote by $\mathcal{M}(\alpha)\subseteq R_{(a,b)}$ the set of monomials $T(p',q')_{(a,b)}$ where $(p',q')\le (p_i,i)$ for some $i=1,\ldots, t.$ 
\end{definition}
Note that for any $ (1,1) \leq (p,q), (p',q') \leq (a+1,b+1)$ then $(p,q)\neq (p',q')$ iff  $T(p,q)_{(a,b)}\neq T(p',q')_{(a,b)}.$ 

\begin{example}\label{ExFer}
	Let be $\alpha:=(3,3,2,1,0)\in \Si(9)^{(2,4)},$ then we can draw $\alpha$ as a Ferrers diagram (see \cite{GVT}, Definition 3.13) with rows and columns labeled 
	$$\alpha:=\begin{array}{l|lllll}
	& 0 & 1 & 2 & 3 & 4 \\
	\hline
	0 & \bullet & \bullet & \bullet & \bullet & \\  
	1 & \bullet & \bullet & \bullet & &\\
	2 & \bullet &  \bullet & & &\\
	\end{array} $$
	Then $\mathcal{M}(\alpha)\subseteq R_{(2,4)}$ is the set of all the monomials $x_1^{i}x_2^{2-i}y_1^jy_2^{4-j}$ where $(i,j)$ is a non-empty entry in the above diagram. 
	
\end{example}
 
\begin{remark}\label{Mbilex}
	For any $\alpha:=(p_1,\ldots, p_{t})\in \Si(\cdot)^{(a+1,b+1)}$ the set of monomials of degree $ {(a,b)}$ not in $\mathcal{M}(\alpha)$ is a bilex set.  	Indeed, let $u:=x_1^{a_1}x_2^{a_2}\in R_{(a,0)}$ and $v:=y_1^{b_1}y_2^{b_2}\in R_{(0,b)}$ be monomials such that $uv\notin M(\alpha),$ i.e. $(a_1+1,b_1+1)\not\le (p_i,i)$ for any $i=1,\ldots,t.$ Given a monomial $u':=x_1^{c_1}x_2^{c_2}\in R_{(a,0)}$ with  $u'> u,$ then $c_1>a_1$ and $(c_1+1,b_1+1)> (a_1+1,b_1+1).$ Therefore $T(c_1+1,b_1+1)_{(a,b)}=u'v\notin \mathcal{M}(\alpha).$ 
\end{remark}

	One can also check that $M(\alpha)$ is a bilex set of monomials with respect to the order $x_2>x_1>y_2>y_1.$  


\begin{definition} \label{M(I)}	
	Let $I\subseteq R$ be a monomial bilex ideal and $(a,b)\in\N^2.$ We denote by $\M_{ab}({I})$ the set of monomial of degree $(a,b)$ not in $I_{(a,b)}.$ 
	
	Moreover we denote by
	$$p_i(\mathcal{M}_{ab}(I)):=\max\left(\{ p'\in \N\ |\ T(p',i)_{(a,b)}\in \M_{(a,b)}(I) \}\cup\{0\}\right) $$ and by 
	$$\alpha_{\M_{ab}(I)}:=\big(p_1(\mathcal{M}_{ab}(I)),\ldots, p_{b+1}(\mathcal{M}_{ab}(I))\big).$$
\end{definition}

\begin{remark}
   	From the definition of monomial bigraded ideal, it is immediate to check that  $\alpha_{\M_{ab}(I)}\in \Si(\cdot)^{(a+1,b+1)}$. Indeed $$a+1\ge p_1(\mathcal{M}_{ab}(I))\ge p_{2}(\mathcal{M}_{ab}(I))\ge \cdots \ge p_{b+1}(\mathcal{M}_{ab}(I)).$$
\end{remark}

\begin{example} 
	Take the bilex ideal minimally generated only in degree $(2,3)$  
	\[I=(x_1x_2y_1^2y_2, x_1x_2y_1^3, x_1^{2}y_1^2y_2, x_1^{2}y_1^3).\]
	Using the notation introduced in Definition \ref{T(p,q)} we write
	\[I=\left(T(2,3)_{(2,3)},T(2,4)_{(2,3)},T(3,3)_{(2,3)},T(3,4)_{(2,3)}\right).
	\]
	Then the set $\mathcal{M}_{(2,3)}(I)$ introduced in Definition \ref{M(I)} is
     \[\mathcal{M}_{(2,3)}(I)=\left\{\begin{array}{cccc}
     T(1,1)_{(2,3)}, &  T(1,2)_{(2,3)}, & T(1,3)_{(2,3)},& T(1,4)_{(2,3)},\\
     T(2,1)_{(2,3)}, & T(2,2)_{(2,3)}, & & \\ 
     T(3,1)_{(2,3)},& T(3,2)_{(2,3)}& &\\
     \end{array}\right\}.\]
      
    Thus we have $\alpha_{\M_{(2,3)}(I)}=(3,3,1,1)\in \Si(\cdot)^{(3,4)}.$
\end{example}

The following result holds.
\begin{lemma}\label{LemmaCap}
	Let $\alpha_1,\alpha_2\in\Si(\cdot)^{(\ell_1,\ell_2)}$ 
	be such that $\alpha_{1}\le \alpha_{2}.$ Then 
	$\mathcal{M}(\alpha_1)\subseteq \mathcal{M}(\alpha_{2}).$
\end{lemma}
\begin{proof} It is trivial.
\end{proof}


\begin{lemma}\label{LemmaH} 
	Let $L$ be a bilex set of monomials of degree $(a,b)$ and let $I:=(L)\subseteq R$ be the ideal generated by the elements in $L.$ 
	Then 
	
	\begin{itemize}
		\item[i)]  $\alpha_{\M_{ab}(I)}\in \Si(H_{R/I}(a,b))^{(a+1,b+1)};$ 
		\item[ii)]  $\alpha_{\M_{a+1b}(I)}=\alpha_{\M_{ab}(I)}^{(1)}$ and $|\M_{a+1b}(I)|=|\M_{ab}(I)|+\lambda_1(\alpha_{\M_{ab}(I)});$ 
		\item[iii)]  $\alpha_{\M_{ab+1}(I)}=\alpha_{\M_{ab}(I)}^{(2)}$ and $|\M_{ab+1}(I)|=|\M_{ab}(I)|+\lambda_2(\alpha_{\M_{ab}(I)}).$ 
	\end{itemize}
\end{lemma}
\begin{proof} 
	Item $i)$ follows from $\M_{ab}(I)\subseteq R_{(a,b)}$ and $H_{R/I}(a,b)=\dim_k R_{(a,b)}- \dim_k I_{(a,b)}= |\M_{ab}(I)|.$ 
	Let $\alpha_{\M_{ab}(I)}=(p_1,\ldots,p_t)\in \Si(\cdot)^{(a+1,b+1)}$ 
	and $\alpha_{\M_{a+1b}(I)}=(p_1',\ldots, p_t')\in \Si(\cdot)^{(a+1,b+1)}.$ Assume $i\in\{1,\ldots, t\}$ such that $p_i'<a+2,$ then
	$p'_i=\max\{q\ | \ T(q,i)_{(a+1,b)}\in \M_{a+1b} \}=\max\{q\ | \ T(q,i)_{(a+1,b)}\notin I_{(a+1,b)} \} \le \max\{q\ | \ T(q,i)_{(a,b)}\notin I_{(a,b)} \}=p_i.$   Since $I$ is only generated in degree $(a,b),$ we get $p_i=p_i'.$
	If $p_i'=a+2,$ then $T(a+2,i)_{(a+1,b)}\notin I_{(a+1,b)}$ and then $T(a+1,i)_{(a,b)}\notin I_{(a,b)}.$ 
	Analogously we get item $iii).$ 
	
	
\end{proof}

\begin{example}
	Let $\alpha:=(3,3,2,1,0)\in \Si(\cdot)^{(2,4)},$ then $\alpha^{(1)}=(4,4,2,1,0)\in \Si(\cdot)^{(3,4)}$ and $\alpha^{(2)}=(3,3,2,1,0,0)\in \Si(\cdot)^{(2,5)}.$ Using Ferrers diagrams as in Example \ref{ExFer} we have 
	\[\alpha^{(1)}:=\begin{array}{l|lllll}
	& 0 & 1 & 2 & 3 & 4 \\
	\hline
	0 & \bullet & \bullet & \bullet & \bullet & \\  
	1 & \bullet & \bullet & \bullet & &\\
	2 & \bullet &  \bullet & & &\\
	3 & \bullet &  \bullet & & &\\
	\end{array} \ \ \ \alpha^{(2)}:=\begin{array}{l|llllll}
	& 0 & 1 & 2 & 3 & 4 & 5 \\
	\hline
	0 & \bullet & \bullet & \bullet & \bullet & & \\  
	1 & \bullet & \bullet & \bullet & & & \\
	2 & \bullet &  \bullet & & & & \\
	\end{array} \]

\end{example}

We are ready to prove the main result of this paper.

\begin{theorem}\label{main}
	Let $H:\N\times\N\to \N$ be a numerical function. Then the following are equivalent
	\begin{enumerate}
		\item[1)] $H$ is a Ferrers function;
		\item[2)] $H=H_{R/I}$ for some bigraded ideal $I\subseteq R.$
	\end{enumerate}  
\end{theorem}
\begin{proof}
		$(1)\to (2)$ Let $H$ be a Ferrers function and let $\{\alpha_{ab}\}\in \Si{(H(a,b))}^{(a+1,b+1)}$ be the set of partitions as required in Definition \ref{DefFerrers}.
		For each $(a,b)\in \N^2$, let $I_{(a,b)}$ be the $k$-vector space generated by monomials of degree ${(a,b)}$ not in $\mathcal{M}(\alpha_{ab}).$ Then we claim that $I:=\oplus_{(a,b)\in\N^2} I_{(a,b)}$ is an ideal of $R.$ Note that, in order to prove the claim, it is enough to show that $(x_1,x_2,y_1,y_2)T\in I$ for any monomial $T\in I.$ 	
		Thus, let $T(p,q)_{(a,b)}\in I_{(a,b)}$ be monomial  of degree $(a,b)$. Say $\alpha_{(a,b)}= (c_1,\ldots, c_b)$, $\alpha_{(a+1,b)}= (c'_1,\ldots, c'_b)$ and  $\alpha_{(a,b+1)}= (c''_1,\ldots, c''_{b+1})$. We collect the relevant facts 
		\begin{itemize}
			\item[$i)$]  $c_q< p$ by $T(p,q)_{(a,b)}\notin \mathcal{M}(\alpha_{ab})$;
			\item[$ii)$] $c'_q\le c_q$ and $c''_q\le c_q$. This follows from $\alpha_{a+1b}\le\alpha_{ab}^{(1)}$, $\alpha_{ab+1}\le\alpha_{ab}^{(2)}$ and $i).$
			\item[$iii)$] $x_1\cdot T(p,q)_{(a,b)}=T(p+1,q)_{(a+1,b)}$ and $x_2\cdot T(p,q)_{(a,b)}= T(p,q)_{(a+1,b)}$ by Definition \ref{T(p,q)};
			\item[$iv)$] $ y_1\cdot T(p,q)_{(a,b)}= T(p,q+1)_{(a,b+1)}$ and $y_2\cdot T(p,q)_{(a,b)}=T(p,q)_{(a,b+1)}$ by Definition \ref{T(p,q)}.
		\end{itemize}
		 
		Assume by contradiction $T(p+1,q)_{(a+1,b)}\notin I$, i.e., $T(p+1,q)_{(a+1,b)}\in \mathcal{M}(\alpha_{a+1,b}).$ Thus, by Definition \ref{T(p,q)} we have $(c'_{q},q)\ge (p+1,q)$, and then $c'_{q}\ge  p+1> c_q.$ This contradicts $ii).$ 
		Analogously, if $T(p,q)_{(a+1,b)}\in \mathcal{M}(\alpha_{a+1,b})$ then $c'_{q}\ge p> c_q$, contradicting $ii).$ 
		In a similar way, by using the inequalities $c_{q}\ge c_{q}'' \ge c_{q+1}'',$ one can show that $T(p,q+1)_{(a,b+1)}\in I$ and  $T(p,q)_{(a,b+1)}\in I.$

		$(2)\to (1)$   Let $I\subseteq R$ be a bilex ideal such that $H_{R/I}=H.$ 
		Then, for any $(a,b)\in \N^2,$ we set $\alpha_{ab}:=\alpha_{\M_{ab}(I)}.$ By item $i)$ in Lemma \ref{LemmaH}, we have $\alpha_{ab}\in \Si(H_{R/I}(a,b))^{(a+1,b+1)}.$ Moreover, by items $ii)$ and $iii)$ in Lemma \ref{LemmaH},  the condition in Definition \ref{DefFerrers} holds since $(x_1,x_2)I_{(a-1,b)}\subseteq I_{(a,b)}$ and $(y_1,y_2)I_{(a,b-1)}\subseteq I_{(a,b)}$.

\end{proof}


\begin{remark}
	Lemma \ref{LemmaH} gives the maximal growths for an Hilbert function  $H$ of a bigraded algebra $R/I.$ A maximal growth for $H$ in degree $(a,b)$ is $(H(a,b)+\lambda_1, H(a,b)+\lambda_2)$ with $(\lambda_1,\lambda_2)$ a maximal element in $\Li{(H(a,b))}^{(a+1,b+1)}.$ The set of maximal growths is not enough to describe the behavior of an Hilbert function. E.g. let $H$ be such that  $$H:=\begin{array}{l|ccccccc}
	& 0 & 1 & 2 & 3 & 4 & 5 & \cdots \\
	\hline
	0 & 1 &  2 &  3 &  4 &  5 & 0 & \cdots \\
	1 & 2 &  4 &  6 &  8 & 10 & 0 & \cdots \\
	2 & 3 &  6 &  9 &  8 &  9 & 0 &\cdots \\
	3 & 4 &  8 &  8 & 10 &  0 & 0  &\cdots \\
	4 & 5 & 10 &  9 &  0 &  0 & 0 &\cdots \\
	5 & 0 &  0 &  0 &  0 &  0 & 0 &\cdots \\
	\vdots & \vdots & \vdots & \vdots & \vdots & \vdots & \vdots & \ddots \\
	\end{array}$$
	To ensure the maximal growth in degree $(2,3)$ and $(3,2)$ we have to take $\alpha_{32}=(4,2,2)$ and $\alpha_{23}=(3,3,1,1).$ 
	Then $\lambda_1(\alpha_{23})=\lambda_2(\alpha_{32})=2$ so $H(3,3)\le \min\{H(2,3)+2, H(3,2)+2 \}=10$ but $\alpha_{23}^{(1)}\cap \alpha_{32}^{(2)} = (4,2,1,1)$ and then $H$ fails to be a Ferrers function.
\end{remark}	

We end this paper with an application of Ferrers functions. Admissible functions were introduced in \cite{GMR} in order to study the Hilbert functions of reduced 0-dimensional schemes in $\popo.$     

\begin{definition}[\cite{GMR}, Definition 2.2]
	Let $H:\N^2\to \N$ be a numerical function and denote by $c_{ij}:=\Delta H(i,j)= H(i,j)+H(i-1,j-1)-H(i-1,j)-H(i,j-1).$
	Then we say that $H$ is an admissible function if
	\begin{enumerate}
		\item $c_{ij}\le 1,$ and $c_{ij}=0$ for $i\gg 0, j\gg 0;$
		\item if  $c_{ij}\le 0$ then $c_{uv}\le 0,$ for all $(u,v)\ge (i,j);$
		\item $0\le\sum_{t=0}^j c_{it}\le \sum_{t=0}^j c_{i-1t}$ and $0\le\sum_{t=0}^i c_{tj}\le \sum_{t=0}^i c_{tj-1}.$
	\end{enumerate}
\end{definition} 

Theorem 2.12 in \cite{GMR} shows that the Hilbert function of a 0-dimensional scheme in $\popo$ is an admissible function.
However the converse fails to be true (\cite{GMR} Example 2.14). Even if, in this paper, we do not worry about the geometrical point of view, it is still interesting to ask if an admissible function is a Ferrers function. Theorem \ref{ThmRg} gives a positive answer to this question.
\begin{theorem}\label{ThmRg}
	If $H$ is an admissible function, then $H$ is a Ferrers function.
\end{theorem}
\begin{proof}Let $(a,b)\in \N$ and  $r\in\{0,\ldots, b\},$ we set  $p_{r+1}^{(ab)}:=\sum_{i=0}^{a} c_{ir}$ and $ \alpha_{ab}:= (p_1^{(ab)}, \ldots,p_{r+1}^{(ab)},\ldots p_{b+1}^{(ab)}   ).$
	From the definition of admissible function, we have $\alpha_{ab}\in \Si{(H{(a,b)})}^{(a+1,b+1)}$ and also $\alpha_{a-1b}^{(1)}\ge\alpha_{ab}$ and $\alpha_{ab-1}^{(2)}\ge\alpha_{ab}.$
\end{proof}

\begin{example}
	Example 2.14 in \cite{GMR} shows an admissible  function that fails to be the Hilbert function of a reduced set of points in $\popo.$  
	$$H:=\begin{array}{l|lllllll}
	& 0 & 1 & 2 & 3 & 4 & 5 & \cdots \\
	\hline
	0 & 1 &  2 &  3 &  4 &  5 & 5 & \cdots \\
	1 & 2 &  4 &  6 &  8 & 10 & 10 & \cdots \\
	2 & 3 &  6 &  8 &  9 & 10 &  10 &\cdots \\
	3 & 4 &  8 & 10 & 10 & 10 &  10 &\cdots \\
	4 & 5 & 10 & 10 & 10 & 10 & 10 &\cdots \\
	5 & 5 & 10 & 10 & 10 & 10 & 10 &\cdots \\
	\vdots & \vdots & \vdots & \vdots & \vdots & \vdots & \vdots & \ddots \\
	\end{array}$$
	
	Meanwhile, as Theorem \ref{ThmRg}, $H$ is a Ferrers function since one can write
	
	$$\alpha_{ij}:=\begin{array}{l|llllll}
	& 0 & 1 & 2 & 3 & 4 & \cdots \\
	\hline
	0 &    (1) &  (1,1) &  (1,1,1) &  (1,1,1,1) &  (1,1,1,1,1) & \cdots \\
	1 &    (2) &  (2,2) &  (2,2,2) &  (2,2,2,2) &  (2,2,2,2,2) & \cdots \\
	2 &    (3) &  (3,3) &  (3,3,2) &  (3,3,2,1) &  (3,3,2,1,1) & \cdots \\
	3 &    (4) &  (4,4) &  (4,4,2) &  (4,4,2,0) &  (4,4,2,0,0)     & \cdots \\
	4 &    (5) &  (5,5) &  (5,5,0) &  (5,5,0,0) &  (5,5,0,0,0)       & \cdots \\
	\vdots  & \vdots & \vdots & \vdots   & \vdots     & \vdots       & \ddots \\
	\end{array}$$
	
	Thus, the associated ideal $I:=(x_1^2y_1^2, x_1x_2y_1^3, x_2^3y_1^3, x_1^4y_2^2, x_1^4y_1y_2, x_1^5, y_1^5  )\subseteq R,$ has Hilbert function $H_{R/I}=H.$
\end{example}


\begin{thebibliography}{99}
\bibitem{ACD}  Aramova A, Crona K, De Negri E. \textit{Bigeneric initial ideals, diagonal subalgebras and bigraded Hilbert functions}. Journal of Pure and Applied Algebra. 2000 Jul 3;150(3):215-35.

\bibitem{BG} Bruns W, Gubeladze J. \textit{Divisorial linear algebra of normal semigroup rings}. Algebras and Representation Theory. 2003 May 1;6(2):139-68.

\bibitem{BH} Bruns W, Herzog HJ. \textit{Cohen-macaulay rings}. Cambridge university press; 1998 Jun 18.

\bibitem{C}
CoCoATeam, \textit{CoCoA: a system for doing Computations
in Commutative Algebra}.
Available at {\tt http://cocoa.dima.unige.it}


\bibitem{GMR} Giuffrida S, Maggioni R, Ragusa A. \textit{On the postulation of $0$-dimensional subschemes on a smooth quadric}. Pacific Journal of Mathematics. 1992 Oct 1;155(2):251-82.

\bibitem{GVT} Guardo E, Van Tuyl A. 
\textit{Arithmetically Cohen-Macaulay sets of points in $\popo.$} SpringerBriefs in Mathematics, Springer 2015.

\bibitem{M}  Macaulay FS. \textit{Some properties of enumeration in the theory of modular systems}. Proceedings of the London Mathematical Society. 1927 Jan 1;2(1):531-55.

\bibitem{PS} Peeva I, Stillman M. \textit{Open problems on syzygies and Hilbert functions}. Journal of Commutative Algebra. 2009;1(1):159-95.

\bibitem{SVT} Sidman J, Van Tuyl A. \textit{Multigraded regularity: syzygies and fat points}. Contributions to Algebra and Geometry. 2006;47(1):1-22.

\bibitem{St}  Stanley RP. \textit{Hilbert functions of graded algebras}. Advances in Mathematics. 1978 Apr 1;28(1):57-83.

\bibitem{TV}  Trung NV, Verma JK. \textit{Hilbert functions of multigraded algebras, mixed multi-plicities of ideals and their applications}. Journal of Commutative Algebra. 2010;2(4):515-65.

\end{thebibliography}
\end{document}